\documentclass[preprint,12pt]{elsarticle}

\usepackage{amssymb}
\usepackage{amsthm}

\usepackage{stackengine}
\usepackage{graphicx}
\usepackage{amsmath}
\usepackage{amsfonts}
\usepackage{physics} 
\usepackage{tikz}
\usepackage{geometry}
\newtheorem{theorem}{Theorem}
\theoremstyle{definition}
\newtheorem{remark}[theorem]{Remark}
\newtheorem{definition}{Definition}

\newtheorem{corollary}{Corollary}

\journal{Journal of Mathematical Physics}

\DeclareUnicodeCharacter{2212}{-}
\begin{document}

\begin{frontmatter}



\title{On Symmetric Pseudo-Boolean Functions: Factorization, Kernels and Applications} 

\author[inst1]{Richik Sengupta}

\affiliation[inst1]{
            addressline={Skolkovo Institute of Science and Technology}, 
            city={Moscow},
            postcode={121205}, 
            country={Russia}}

\author[inst2]{Jacob Biamonte}
            \affiliation[inst2]{
            addressline={Yanqi Lake Beijing Institute of Mathematical Sciences and Applications}, 
            city={Beijing},
            postcode={101408}, 
            country={China}} 

\begin{abstract}
A symmetric pseudo-Boolean function is a map from Boolean tuples to real numbers which is invariant under input variable interchange.  We prove that any such function can be equivalently expressed as a power series or factorized and obtain the relationship between the coefficients expressed in different forms.  The kernel of a pseudo-Boolean function is the set of all inputs that cause the function to vanish identically.  Any $n$-variable symmetric pseudo-Boolean function $f(x_1, x_2, \dots, x_n)$ has a kernel corresponding to at least one hyperplane given by a constraint $\sum_{l=1}^n x_l = \lambda$ for $\lambda\in \mathbb{C}$ constant.  We use these results to analyze symmetric pseudo-Boolean functions appearing in the literature of spin glass energy functions (Ising models), quantum information and tensor networks.
\end{abstract}

\end{frontmatter}

\section{Introduction}

Pseudo-Boolean functions $f:\mathbb{B}^n\rightarrow \mathbb{R}$ have wide applications in applied mathematics including areas such as game theory, electrical engineering and decision theory. The term pseudo-Boolean function also arises where $f: \{-1,1\}^{n}\rightarrow  \{-1,1\}$ in the Fourier analysis of Boolean functions~\cite{Odn14}. 

Pseudo-Boolean functions are mathematically equivalent to the generalized or tunable Ising models appearing broadly in statistical physics,  classical and quantum annealing. The  pseudo-Boolean functions \cite{BH02} can be expressed as a sum of disjoint variable products: 
\begin{equation}\label{eqn:canonical}
    f({\bf x}) = a_0 \overline{x}_1 \dots \overline{x}_n + a_1 x_1 \overline{x}_2\dots \overline{x}_n+\dots + a_N x_1\dots x_n.
\end{equation}
We use ${\bf x}$ to represent the variable vector $(x_1, x_2, \dots, x_n)$, $\forall l, a_l\in \mathbb{R}$, concatenated variables such as $x_k x_m$ are multiplied with the multiplication symbol omitted and finally, $\overline{x}$ is the logical negation of Boolean variable $x$.  We can replace all $\overline{x}\mapsto 1-x$ in \eqref{eqn:canonical}.  After some calculation, one arrives at the canonical form \cite{BH02}, 
\begin{equation}\label{eqn:canonical2}
\begin{split}
        f({\bf x}) = c_0 + c_1 x_1 + c_2x_2 +\dots +c_n x_n +   \\
         c_{1,2}x_1x_2+\dots +c_{n-1, n}x_{n-1}x_n + \\
         \vdots \\ 
          +c_{1, 2,\dots n}x_1x_2\dots x_n. 
\end{split}
\end{equation}
These forms, \eqref{eqn:canonical}, \eqref{eqn:canonical2} each uniquely define any $n$-variable pseudo-Boolean function in terms of $N=2^n$ real numbers $\bf a$ and $\bf c$.   If for any input ${\bf y}\in \{0,1\}$, $f(\bf y)\geq 0$ then $f(\bf x)$ is called non-negative.  We call 
\begin{equation}\label{eqn:af}
    c_0 + c_1 x_1 + c_2x_2 +\dots + c_n x_n 
\end{equation}
the affine part of $f({\bf x})$. 
A pseudo-Boolean function is called almost-positive if all the coefficients in \eqref{eqn:canonical2}, except possibly those in the affine part \eqref{eqn:af}, are non-negative \cite{CHH+89}. Note that in the Fourier analysis of Boolean functions \cite{Bel+96}, the characters
\begin{equation}
    \chi({\bf x}) = c(S)\prod_{j\in S} x_j
\end{equation}
for $S\subseteq [n]$ and $c(S)$ constant are sometimes called linear as they satisfy $\chi({\bf x} {\bf y}) = \chi({\bf x})\chi({\bf y})$. Whereas we would call \eqref{eqn:af} linear (affine) for $c_0=0$ ($c_0\neq 0$) as it defines an $n$-hyperplane.  

A quadratic pseudo-Boolean function is given as
\begin{equation}\label{eqn:quadratic}
\frac{f({\bf x})}{x_ix_jx_k=0, \forall i\neq j \neq k}  = c_0 + c_1 x_1 + c_2x_2 +\dots +c_n x_n +
    \dots+ c_{12}x_1x_2+\dots +c_{n-1, n}x_{n-1}x_n, 
\end{equation}
where the fraction on the left-hand-side denotes $f(\bf x)$ modulo the constraint in the denominator. In other words the quadratic pseudo-Boolean function is a non-homogeneous polynomial of degree two. The minimisation of pseudo-Boolean functions is {\sf NP}-hard. We call the set of inputs such that $f({\bf x})=0$ the kernel of the pseudo-Boolean function $f$.  Deciding if a (quadratic) pseudo-Boolean function has a non-trivial kernel is {\sf NP}-complete.  

The canonical form \eqref{eqn:canonical2} appears throughout the literature on pseudo-Boolean functions \cite{BH02}. A large body of research is devoted to adding additional slack variables to reduce pseudo-Boolean functions to quadratic pseudo-Boolean functions (called quadratization) \cite{BG12}.   

We let $\sigma$ denote any permutation of the coordinates of a variable vector $(x_1, x_2, \dots, x_n)$.  We call the pseudo-Boolean functions 
\begin{equation}
    f({\bf x}) = f(\sigma {\bf x})
\end{equation}
symmetric.  In this setting, we derive two alternative canonical forms, one of which appears as a product of roots and provides an intriguing relationship between algebra and symmetric pseudo-Boolean functions.  In terms of past work, Anthony,  Boros, Crama and Gruber \cite{ABCG16} proved that any symmetric pseudo-Boolean function of $n$-variables can be quadratized by adding at most $n-2$ additional slack variables. With some caveats, Boros, Crama and Rodr{\'i}guez-Heck recovered logarithmic scaling \cite{BCR18}.  Marichal and Mathonet develop an index for measuring the influence of a pseudo-Boolean variable and consider symmetric approximations of pseudo-Boolean functions~\cite{MM12}.  Symmetric Boolean functions and their representations as polynomials that agree on the points $\{0,1\}^n$ also appear in the literature \cite{CV05, Weg87}. Similar problems were studied in \cite{D19}. The present work provides a reformulation method to expand, factor and characterise the kernel of any symmetric pseudo-Boolean function.

To explain the applications of the results we derive in this paper, let us consider some examples.  In quantum information science and in tensor networks, one often deals with specific tensors and specific quantum states that are symmetric under variable permutation.  
For example, the GHZ state (named after Greenberger, Horne and Zeilinger) is written as a vector
\begin{equation}
    \frac{e_{000}+e_{111}}{\sqrt 2}. 
\end{equation}
Here, $e_{kkk}= e_{k} \otimes e_{k} \otimes e_{k},   k \in \{0,1\}$ denotes the tensor product of the standard basis vectors $e_0 = (1,0)$ and $e_1 = (0,1)$, where the product lies in the tensor product space $[{\mathbb{C}}^2]^{\otimes 3}$. 

Since, a state is represented by a vector and a vector is uniquely determined by its coefficients in a basis expansion, one might equivalently associate to such a state a Kronecker's delta function on three binary indices, $\delta_{l m n}$ which satisfies the equation 
\begin{equation}\label{eqn:delta}
    \delta_{000} = \delta_{111} = 1
\end{equation}
and vanishes otherwise---hence $\delta_{l m n}$ has a 6 element kernel when considering the set $\{0,1\}^3$.  One would expand \eqref{eqn:delta} as a pseudo-Boolean function in the evident way, as the sum of two disjoint terms: 
\begin{equation}\label{delt}
    \delta_{k l m} = k l m  + (1-l)(1-m)(1-k) = 1 - l - m - k + m k + kl + lm 
\end{equation}
We present a general theorem which shows that a function such as the one given by (\ref{delt})  would be equivalently expressed as 
\begin{equation}
\begin{split}\label{delt2}
        \delta_{l m k} = \frac{1}{2} (l + m + k)^2 - \frac{3}{2}(l + m + k) + 1   \\
        =\frac{1}{2} (l + m + k -2) (l + m + k -1).
\end{split}
\end{equation}
We will see that the root system $\{1, 2\}$ provides a classification of (\ref{delt2}) where the function vanishes identically whenever 
\begin{equation}
    \begin{split}\label{eqn:delker} 
        l + m + k = 2, \\
        l + m + k = 1. 
    \end{split}
\end{equation}
Together the equations \eqref{eqn:delker} have six solutions over $\{0,1\}^3$ whereas the entire solution space corresponds to the two hyperplanes  $l+m+k=1,2$.  

Likewise,  one would express the $\sf XOR$ function as 
\begin{equation} 
\begin{aligned}
  \left[\bigoplus\right](l, m, k) &= l\oplus m\oplus k = \frac{2}{3}(l + m + k)^3-3(l + m + k)^2 +\frac{10}{3}(l + m + k) \\
  &= \frac{2}{3}(l + m + k - 2)(l + m + k )(l + m + k - \frac{5}{2}).
  \end{aligned}
\end{equation} 
Interestingly, we also arrive at what we call, {\it fractional roots}. These occur whenever a root (such as $5/2$) appears outside of the Boolean input domain. In this case, the kernel over the Boolean domain $\{0,1\}^3$ corresponds to only four elements, whereas all solutions correspond to the three hyperplanes $l+m+k=0, 2, 5/2$. 

Symmetric pseudo-Boolean functions also appear when considering the symmetric Ising model from statistical and quantum mechanics\cite{H21}.  This model finds applications in physical annealing devices\cite{GWPN09}. One will program the model parameters and utilize a physical process (such as quantum-assisted or classical cooling) to minimize the model to solve a computational problem instance \cite{WFB12}. The symmetric Ising model is typically considered with variables $z = \pm1$ and with the affine transformation ($z=1-2x$) can be written up to a constant as follows
\begin{equation}\label{eqn:ham}
    H({\bf x}) = \frac{J}{2} \sum_{l\neq m} x_lx_m  + h \sum_l x_l. 
\end{equation}
In (\ref{eqn:ham}) the coupling strength $J$ is non-zero (otherwise \eqref{eqn:ham} is trivial) and the bias $h$ takes only real values. Using our methods, \eqref{eqn:ham} is equivalently expressed as 
\begin{equation}
    H({\bf x}) = \frac{J}{4}\left(\sum_{l=1}^n x_l + \frac{4h}{J} -1\right) \sum_{l=1}^n x_l.   
\end{equation}
The quantity $\sum_{l=1}^n x_l$ corresponds to the total $Z$-spin for any state of the system and its expected value vanishes whenever,
\begin{equation}
    \begin{split}
        \sum_{l=1}^n x_l = 0,~~ \sum_{l=1}^n x_l = 1-\frac{4h}{J}. \\
            \end{split}
\end{equation}
and hence,  states restricted to either of these two hyperplanes correspond to the kernel of $H({\bf x})$. We will explore caveats related to the fractional roots.

Our results arise from establishing a mapping from any symmetric multivariate function with idempotent indeterminates to the algebra of polynomials in one variable. We will proceed by first establishing further properties of pseudo-Boolean functions in \S~\ref{sec:pbf}. Here we recall the foundational theorem from \cite{BP89, Ros72} which establishes that pseudo-Boolean functions take extreme values on Boolean inputs under the domain $[0,1]^n$.  We continue in \S~\ref{sec:roots} by establishing the root system and factorisations developed in this study.  We conclude in \S~\ref{sec:examples} with several examples.  

\section{Pseudo-Boolean Functions}
\label{sec:pbf}

There are multiple ways to represent \ (pseudo) boolean functions. A common method is to use a a logical basis with the $\displaystyle \land ,\ \lor ,\ \neg \ $operators giving rise to the famous normal forms like Conjuctive Normal Form, Disjunctive Normal Form etc. 

Instead of using a logical basis, we will be interested in a polynomial basis to represent pseudo-Boolean functions.

Consider the following notation:
\begin{equation}
    {\bf x}^I \equiv (x_1)^{i_1} (x_2)^{i_2} \cdots (x_n)^{i_n}
\end{equation}
\begin{equation}
    a_{\bf I} \equiv a_{i_1i_2\cdots i_n}
\end{equation}
where $I= (i_1,i_2,\ldots,i_n) \in \{0, 1\}^n$.

\begin{definition}
A pseudo-Boolean function has type $$\{0,1\}^n\to \mathbb{R}$$ and can be uniquely written as 
\begin{equation}\label{eqn:fstandard}
    f({\bf x}) = \sum_{I\in \{0, 1\}^n} a_I {\bf x}^I, 
\end{equation}
where $a_I \in \mathbb{R}$ and $x_i \in \{0,1\}$ are indeterminate. 
\end{definition}

If we  let $(x_i)^1=x_i$ and $(x_i)^0=\Bar{x_i},$ this equates \eqref{eqn:canonical} and \eqref{eqn:fstandard} where ${\bf x}^I$ will form a basis of the $2^n$ dimensional vector space of polynomials ${\bf P}$. We further remind the reader that we let $0^0=1.$
\begin{definition}
A pseudo-Boolean function has type $$\{0,1\}^n\to \mathbb{R}$$ and can be uniquely written as a multi-linear polynomial: 
\begin{equation}\label{eqn:fstandard2}
    f(x)=\sum _{I\subseteq [n]}{c}(I)\prod _{i\in I}x_{i}, 
\end{equation}
where $c(I) \in \mathbb{R}~ \textrm{and}~[n]=\{1,...,n\}$.
\end{definition}

This is equivalent to the "graded" form \eqref{eqn:canonical2}, also note $c(\emptyset)=c_0\in \mathbb{R}$. We will use this form further.

It is interesting to note that the extremal points of pseudo-Boolean functions do not change even if the input domain of the function is changed from $\{0,1\}^n$ to $ [0,1]^{n}.$

It was established in the following theorem by Boros and Prékopa \cite{BP89}, Rosenberg \cite{Ros72} and reviewed by Boros and Hammer \cite{BH02}. 

\begin{theorem}[Extrema at Boolean inputs~\cite{BP89, Ros72}]\label{thm:extrema}
Consider a pseudo-Boolean function $f(x_1,x_2,\ldots,x_n)$ and let ${\bf r} \in [0,1]^{n}$. Then there exist binary vectors
$ {\bf x}; {\bf y} \in \{0,1\}^n$ for which
$f({\bf x}) \leq f({\bf r}) \leq f({\bf y}).$ 
\end{theorem}

\begin{proof}
Since, $f$ is a continuous function and $[0,1]^n$ is a compact subset of $\mathbb{R}^n$, from Weierstrass's extreme value theorem, it follows that maxima and minima of this function are attained on $[0,1]^n.$
Thus, to prove the theorem it suffices to prove that maxima and minima of $f({\bf s})$ can be attained for ${\bf s} \in \{0,1\}^n. $ 
Consider a point 
\begin{equation}
\hat{\bf x}=(\hat{x}_1,\hat{x}_2, \ldots,\hat{x}_n ) \in \arg\max{f}(\text{or} \arg\min{f}) .
\end{equation}
Let us pick an arbitrary coordinate $\hat{x}_j$ of $\hat{\bf x}$ and keep the other optimal coordinates fixed. Express $f =m\hat{x}_j+c,$ where $m,c \in \mathbb{R}$ i.e.~$f$ is a one-dimensional hyperplane as a function of $\hat{x}_j$. Evidently, if $m \neq 0$ depending on the $sgn(m)$ the maxima (minima) of $f$ is achieved if $\hat{x}_j\in \{0,1\}$ as maxima (minima) of a monotonic function is attained at the endpoints on which it is defined. If $m=0,$ given all other coordinates of the vector $\hat{x}$ are optimal, $\hat{x}_j$ can be replaced arbitrarily with $0$ or $1$ as it does not change the value of $f.$  Since, $\hat{x}_j$ was chosen arbitrarily, $\hat{x}_j\in \{0,1\}$ for all $j\in \{1,2,\ldots,n\}$ will ensure maxima
(minima).

\end{proof}

\section{Root Systems and Factorisation}
\label{sec:roots}

By $\sigma{\bf (x)}:\mathbb{B}^n \to \mathbb{B}^n$  we denote a permutation of the elements of the tuple ${\bf x} \in \mathbb{B}^n.$

\begin{definition}\label{def:sympbf}
A pseudo-Boolean function $f:\{0,1\}^n\to \mathbb{R}$ is called symmetric if for all ${\bf x} \in \mathbb{B}^n \ f({\bf x})=f(\sigma{\bf (x)})$ for all $\sigma{\bf (x)}.$ 
\end{definition}

\begin{remark}
 The authors in \cite{ABCG16, MM12} tailor equivalent definitions to the one given above in Definition \ref{def:sympbf}.  The authors \cite{ABCG16} state equivalently that, a symmetric pseudo-Boolean function $f:\{0,1\}^n\mapsto \mathbb{R}$ depends only on the Hamming weight of the input.  The Hamming weight of a bit vector ${\bf x}$ is given by the 1-norm
    \begin{equation}
        \|{\bf x}\|_1 = \sum_{l=1}^n x_l
    \end{equation}
and hence, a pseudo-Boolean function is symmetric whenever there exists a discrete function $k:\{0, 1, \dots, n\}\mapsto \mathbb{R}$ such that $f({\bf x}) = k(\|{\bf x}\|_1)$. 
\end{remark}

\begin{remark}
From \eqref{eqn:canonical}, it follows that the canonical form of a symmetric pseudo-Boolean function is  
\begin{equation}\label{eq:canonspb}
    f({\boldsymbol {x}})=a_0+a_1\sum _{i_1}x_{i_1}+a_{2}\sum_{i_1<i_2}x_{i_1}x_{i_2}+a_{3}\sum _{i_1<i_2<i_3}x_{i_1}x_{i_2}x_{i_3}+\ldots\\
     + a_{n}\sum _{i_1<i_2<i_3< \ldots <i_n}x_{i_1}x_{i_2}x_{i_3}\ldots x_{i_n}
\end{equation}
 
\end{remark}

Beginning with a series expansion, we now derive some results specific for symmetric pseudo-Boolean functions.  

By $\{\substack{j \\ i}\}$
we denote the Stirling number of the second kind i.e.~the number of ways to partition a set of $j$ elements into $i$ nonempty subsets.

\begin{theorem}[Series expansion]\label{series}
Let  $f:\mathbb{B}^n\to \mathbb{R}$ be a pseudo-Boolean function in canonical form \eqref{eq:canonspb}, ${\bf a}=(a_1,a_2,\ldots,a_n)$  and $c_0=a_0$ then there exists a unique ${\bf c}=(c_1,c_2,\ldots,c_n)\in \mathbb{R}^n$ and such that 
\begin{equation}\label{summ}
f({\bf x})=\sum_{l=0}^nc_l\left(\sum_{k=1}^{n}x_k\right)^l,
\end{equation}
 \begin{equation}
   {\bf a}= B{\bf c},
 \end{equation}
 where 
 \begin{equation}
     B =\displaystyle \begin{pmatrix}
1! \{\substack{1 \\ 1}\} & 1! \{\substack{2 \\ 1}\} & 1! \{\substack{3 \\ 1}\} & \dotsc  & 1! \{\substack{n \\ 1}\}\\\\
0 & 2! \{\substack{2 \\ 2}\} & 2! \{\substack{3 \\ 2}\} & \dotsc  & 2! \{\substack{n \\ 2}\}\\\\
0 & 0 & 3! \{\substack{3 \\ 3}\} & \dotsc  & 3! \{\substack{n \\ 3}\}\\
\vdots  & \ddots  & \ddots  & \ddots  & \vdots\\
0 & 0 & \ldots & 0 & n! \{\substack{n \\ n}\}
\end{pmatrix}.
\end{equation}
\end{theorem}

\begin{proof}
Let us consider the symmetric pseudo-Boolean function  

\begin{equation}\label{symm}\small
\begin{split}
    f({\boldsymbol {x}})=a_0+a_1\sum _{i_1}x_{i_1}+a_{2}\sum_{i_1<i_2}x_{i_1}x_{i_2}+a_{3}\sum _{i_1<i_2<i_3}x_{i_1}x_{i_2}x_{i_3}+\ldots\\
     + a_{n}\sum _{i_1<i_2<i_3< \ldots <i_n}x_{i_1}x_{i_2}x_{i_3}\ldots x_{i_n}
\end{split}
\end{equation} 
 where $i_k \in \{1,2,\ldots,n\}.$ 
 Using idempotency $x^2=x$ and equating (\ref{symm}) to (\ref{summ}) we obtain:
 \begin{equation}
     a_0=c_0;
     a_r = \beta_{r,1}c_r+\beta_{r,2}c_{r+1}+\ldots+ \beta_{r,1+(n-r)}c_n.
 \end{equation}

The matrix of coefficients takes the form:

\begin{equation}
     \beta =\displaystyle \begin{pmatrix}
\beta_{11} & \beta_{12} & \beta_{13} & \dotsc  & \beta_{1n}\\
0 & \beta_{21} & \beta_{22} & \dotsc  & \beta_{2,n-1}\\
0 & 0 & \beta_{31} & \dotsc  & \beta_{3,n-2}\\
\vdots  & \ddots  & \ddots  & \ddots  & \vdots\\
0 & 0 & \ldots & 0 & \beta_{n1}
\end{pmatrix}.
\end{equation}

By counting the number monomials of fixed degree from (\ref{summ}) we obtain that each entry of $\beta$ is a sum of  multinomial coefficients with positive $k_i$ i.e.
 \begin{equation}
\beta_{r,j}=\sum_{k_i>0, k_1+k_2+\ldots+k_r=r+j-1}{\frac {(r+j-1)!}{k_{1}!\,k_{2}!\cdots k_{r}!}}.     
 \end{equation}
 In vectorized form the equation can be rewritten as:   
 \begin{equation}
   {\bf a}= B{\bf c},
 \end{equation}
where 
\begin{equation}
(B)_{i,j}=\beta_{i,j-i+1}=\sum_{k_i>0, k_1+k_2+\ldots+k_i=j}{\frac {j!}{k_{1}!\,k_{2}!\cdots k_{i}!}=\sum_{l=0}^i (-1)^l {(\substack{i \\ l})}(i-l)^j}=i! \{\substack{j \\ i}\}.
\end{equation}


Thus,
 \begin{equation}
     B =\displaystyle \begin{pmatrix}
1! \{\substack{1 \\ 1}\} & 1! \{\substack{2 \\ 1}\} & 1! \{\substack{3 \\ 1}\} & \dotsc  & 1! \{\substack{n \\ 1}\}\\\\
0 & 2! \{\substack{2 \\ 2}\} & 2! \{\substack{3 \\ 2}\} & \dotsc  & 2! \{\substack{n \\ 2}\}\\\\
0 & 0 & 3! \{\substack{3 \\ 3}\} & \dotsc  & 3! \{\substack{n \\ 3}\}\\
\vdots  & \ddots  & \ddots  & \ddots  & \vdots\\
0 & 0 & \ldots & 0 & n! \{\substack{n \\ n}\}
\end{pmatrix}.
\end{equation}
Here, $\{\substack{j \\ i}\}$ is the the Stirling number of the second kind i.e.~the number of ways to partition a set of $j$ elements into $i$ nonempty subsets.
The matrix $B$ is upper triangular. Hence, the unique values for ${\bf c} \in \mathbb{R}^n$ can be readily obtained using substitution.

\end{proof}

\begin{corollary}[Product factorisation]
 A symmetric pseudo-Boolean function $f:\mathbb{B}^n\to \mathbb{R}$ of degree $n$   can be expressed as
 \begin{equation}\label{fact}
      f({\bf x})=K\prod_{l=1}^n\left(\lambda_l-\sum_{k=1}^{n}x_k\right), ~~\lambda_l, K \in \mathbb{C}.
\end{equation}
\begin{proof}
 Substituting $\sum_{k=1}^{n}x_k=X$ into (\ref{summ}) we obtain a polynomial $Q(X)$ of degree $n$ in one variable. Since, $\mathbb{C}$ is algebraically closed, this polynomial will have $n$ roots $\lambda_l \in \mathbb{C}.$ 
\end{proof}
\end{corollary}

\begin{remark}
    Note that in equation (\ref{fact}) even though $\lambda_l \in \mathbb{C},$ only real values are obtained when $x_i\in \mathbb{B}$  which matches with the corresponding values of the unfactored symmetric pseudo-Boolean function. On another note, using non-binary real or complex values as inputs may help to provide insight on properties of the given symmetric pseudo-Boolean function much like the properties (e.g. the radius of convergence)  of real taylor series like $\frac{1}{x^2+1}$ can be understood by passing to the complex domain.
\end{remark}

\begin{remark}
We note that $f({\bf x})$ vanishes identically for inputs constrained under \eqref{fact} as 
\begin{equation}
   \sum_{k=1}^{n}x_k =  \lambda, 
\end{equation}
defining the kernel of $f({\bf x})$. 
\end{remark}

\begin{remark}
  It is interesting that whereas the Hamming weight can be used to define any symmetric pseudo-Boolean function when considering inputs ${\bf x}\in \{0,1\}^n,$ considering the corresponding polynomial
$ Q(X)$ over $\mathbb{C}$ serves to classify the kernel of $f({\bf x})$.  
\end{remark}

\section{Examples}
\label{sec:examples} 

Let us consider a few examples.  We can write the delta function on $k$ binary variables (denoted $\bf m$) in its product form viz.,  
\begin{equation}
    \delta_{\bf m} = \frac{(-1)^{k
    -1}}{ (k-1)!} \prod_{l=1}^{k-1} \left[\sum_{l=1}^k x_l - l\right]. 
\end{equation}

Here, $\delta_{\bf m}$ has $k-1$ roots $\{1,2,3,\ldots,k-1\}$ as the symmetric form is of degree $k-1$ and with the kernel corresponding to $k$ long bit vectors  from $\{0,1\}^n\setminus \{0^{\times k},1^{\times k}\}.$

It can be obtained that the symmetric form of the XOR function (addition modulo 2) of $n$-variables has the form 
\begin{equation}\label{XOR symm}\small
\begin{split}
    f({\boldsymbol {x}})=\sum _{i_1}x_{i_1} -2\sum_{i_1<i_2}x_{i_1}x_{i_2}+4\sum _{i_1<i_2<i_3}x_{i_1}x_{i_2}x_{i_3}+\ldots\\
     + (-2)^n\sum _{i_1<i_2<i_3< \ldots <i_n}x_{i_1}x_{i_2}x_{i_3}\ldots x_{i_n}. 
\end{split}
\end{equation}
One might factor the symmetric form using  Theorem \ref{series} but the roots obtained are not always aesthetically pleasing.

As mentioned in the introduction, symmetric pseudo-Boolean functions also arrive when considering the symmetric Ising model from statistical and quantum mechanics.  In \eqref{eqn:ham} we established that 
\begin{equation}\label{eqn:ham2}
    H({\bf x}) = \frac{J}{2} \sum_{l\neq m} x_lx_m  + h \sum_l x_l = \frac{J}{4}\left(\sum_{l=1}^n x_l + \frac{4h}{J} - 1\right) \sum_{l=1}^n x_l.  
\end{equation}
The quantity $\sum_{l=1}^n x_l$ corresponds to the total spin along the Z-axis for any state of the system, the non-zero coupling strength $J$ sets the interaction strength and $h$ is called the local bias.  

To perform further analysis, we will lift our polynomial to a matrix embedding.  We consider the vector space $[\mathbb{C}^2]^{\otimes n} \cong \mathbb{C}^{2^n}$ where $n$ corresponds to the total variable number appearing in \eqref{eqn:ham2}.  

We will recall the notation $e_{{\bf x}}= e_{x_1} \otimes e_{x_2} \otimes\cdots \otimes  e_{x_n}$.  Here, $\forall l, x_l \in \{0,1\}$ this expression denotes the tensor product of the standard basis vectors $e_0 = (1,0)$ and $e_1 = (0,1)$ where the product lies in the tensor product space $[\mathbb{C}^2]^{\otimes n}$.  We will follow Dirac's notation and replace vectors $e_{{\bf x}}$ with $\ket{\bf x}$ and covectors $(e_{{\bf x}})^\dagger$ with $\bra{\bf x}$.  Then a general quantum state can be written as 
\begin{equation}
    \ket{\psi}= \sum_{{\bf x}\in\{0,1\}^n} c_{\bf x}\ket{\bf x}
\end{equation}
which is a unit vector in $[\mathbb{C}^2]^{\otimes n}$ when 
\begin{equation}
    \braket{\psi}{\psi} = \sum_{{\bf x}\in\{0,1\}^n} |c_{\bf x}|^2 = 1. 
\end{equation}
We then note that 
\begin{equation}
    \ket{\psi} \in [\mathbb{C}^2]^{\otimes n} \cong \mathbb{C}^{2^n} \cong \text{span}_\mathbb{C}\{ \ket{\bf x}| {\bf x}\in \{0,1\}^n\}. 
\end{equation}
We lift any input vector ${\bf x}\in \{0,1\}^n$ to a normalized vector as
\begin{equation}\label{eqn:r}
 {\bf x} \mapsto \ket{\bf x}.  
\end{equation}

We then consider the  embedding of the pseudo-Boolean function $H({\bf x})$ 
 expressed in form \eqref{eqn:canonical} into the vector space of diagonal matrices :  
\begin{equation}
    H({\bf x}) \mapsto {\hat H} = \sum_{{\bf I}\in\{0,1\}^n} H(I) \ket{\bf I}\bra{\bf I} \in \text{diag}~\text{mat}_\mathbb{R}(2^n). 
\end{equation}

Here, $H(I)$ is the value of the function $H({\bf x})$ on the boolean input tuple $I$ and $\ket{\bf I}\bra{\bf I}$ is the corresponding outer product.

 We can construct an inverse (called, flattening) onto the space of polynomials  as 
\begin{equation}
    H'({\bf x}) = \sum_{I\in \{0, 1\}^n} \bra{I} {\hat H} \ket{I} {\bf x}^I, 
\end{equation} 
and we see that $\forall {\bf y}\in \{0,1\}^n$, $H'({\bf y})=H({\bf y})$ and on binary tuples   
\begin{equation}
    H({\bf I}) = \bra{{\bf I}}\hat H \ket{{\bf I}}. 
\end{equation}

Hence, we provided a matrix form of pseudo-Boolean function which finds application in spin-glass minimisation. When considering the linear embedding of the symmetric Ising model we see that the operator has a kernel over binary inputs for $\sum_{l=1}^n x_l=0$ and whenever

\begin{equation}
    \frac{h}{J}=\frac{k}{4},~ k \in [-(n-1),1].
\end{equation}

\section{Conclusion}

The authors anticipate that extensions of these results would proceed in several ways.  To prove existence of the roots required both symmetry of the basis ($\sum_{l=1}^n x_l$) and the idempotent property of the indeterminates ($x^2=x$).  These properties could be generalized to appear in another setting of symmetric and idempotent algebras.  

In terms of the basis, the quantity $\sum_{l=1}^n x_l$ appears in physics as the total spin of a system along the $Z$-axis.  When the variables are restricted to the Boolean domain, it becomes the Hamming weight which classifies symmetric Boolean functions in terms of maps from $[0:1+n]$ to $\{0, 1\}$.  Symmetric pseudo-Boolean functions $f({\bf x})$ are  classified as maps defined on $[0:1+n]$ by the existence of a function $m(|{\bf x}|)$ where $|{\bf x}|$ is the Hamming weight and $m(|{\bf x}|) = f({\bf x})$~\cite{ABCG16}.  Considering Theorem \ref{thm:extrema}, it is shown by the authors \cite{BP89, Ros72} that for any ${\bf r}\in [0,1]^n$ and pseudo-Boolean function $f$, there exist bit strings ${\bf x}, {\bf y}$ such that $f({\bf x}) \leq f({\bf r}) \leq f({\bf y})$. The intermediate case where ${\bf r}\in [0,1]^n$ appears in this work as the roots ($\lambda$) of symmetric pseudo-Boolean functions where shown to not always correspond to solutions of the equation $\sum_{l=1}^n x_l = \lambda$ for ${\bf x}\in \{0,1\}^n$.  Indeed, ${\bf x}$ can not only take algebraic values, but in general, complex values.  

Considering further the basis, consider the delta function on two binary inputs $x, y$: 
\begin{equation}
    \delta_{x,y} = 1 - (x-y)^2.  
\end{equation}
We see that $\delta_{x,y} = \delta_{y,x}$ where the basis elements $\{1, (x-y)^2\}$ are both permutation symmetric and clearly differ from the basis $\{1, x+y, (x+y)^2\}$ considered in the paper.  The function nonetheless still factors as 
\begin{equation}
    \delta_{x,y} = \left(1-(x-y)\right)\left( 1+(x-y) \right). 
\end{equation}
This example shows that the factorisation is not restricted to the basis functions
\begin{equation}
    \left\{\left(\sum_{l=1}^n x_l\right)^k|k\in [0:n]\right\}. 
\end{equation}

\section{Acknowledgement} 
We thank Soumik Adhikary, Will Donovan and Akshay Vishwanathan for various discussions on pseudo-Boolean functions.  The authors thank Nike Dattani for mentioning to us that sublinear scaling in the quadratization of symmetric pseudo-Boolean functions was first recovered by Boros, Crama and Rodr{\'i}guez-Heck.  

\bibliographystyle{alpha} 
\bibliography{refs}

\end{document}